\documentclass[a4pitaper,11pt]{amsart}
\usepackage{amsmath,amsthm,amssymb}
\usepackage{xcolor}
\usepackage[mathscr]{eucal}
 \usepackage{cite}
 \usepackage{bm} 
\usepackage{upgreek}
\usepackage[bookmarks,bookmarksnumbered]{hyperref}
\usepackage{enumerate}

\numberwithin{equation}{section}

\newcommand{\car}{\curvearrowright}

\theoremstyle{plain}
\newtheorem{main}{Theorem}
\newtheorem{mcor}[main]{Corollary}

\newtheorem{theorem}{Theorem}[section]

\newtheorem{lemma}[theorem]{Lemma}
\newtheorem{proposition}[theorem]{Proposition}
\newtheorem{corollary}[theorem]{Corollary}
\theoremstyle{definition}

\newtheorem{definition}[theorem]{Definition}
\newtheorem{example}[theorem]{Example}
\newtheorem{notation}[theorem]{Notation}
\newtheorem{remark}[theorem]{Remark}

\begin{document}

\title[Examples of W$^*$ and C$^*$-superrigid product groups]
{Examples of W$^*$ and C$^*$-superrigid product groups}

\author{Jakub Curda, Daniel Drimbe}
\address{}
\email{}
\thanks {}

\begin{abstract} 
We provide a new large class $\mathcal C_{AFP}$ of amalgamated free product groups for which the product rigidity result from \cite{CdSS15} holds: if $G_1,\dots,G_n\in\mathcal C_{AFP}$ and $H$ is any group such that $L(G_1\times\dots\times G_n)\cong L(H)$, then there exists a product decomposition $H=H_1\times\dots\times H_n$ such that $L(H_i)$ is stably isomorphic to $L(G_i)$, for any $1\leq i\leq n$. The class $\mathcal C_{AFP}$ contains $W^*$ and $C^*$-superrigid groups from \cite{CD-AD20}. Consequently, we obtain examples of product groups that are both $W^*$ and $C^*$-superrigid.
\end{abstract}

\maketitle

\section{Introduction}

For a countable group $G$, we denote by $\ell^2(G)$ the Hilbert space of all square-summable complex functions on $G$, and by $\mathbb B(\ell^2(G))$ the algebra of bounded linear operators on $\ell^2(G)$. The complex group algebra $\mathbb C[G]$ can be naturally realized as a $*$-subalgebra of $\mathbb B(\ell^2(G))$ via the left convolution action on $\ell^2(G)$.
The {\it reduced group $C^*$-algebra $C_r^*(G)$} and {\it the group von Neumann algebra $L(G)$} of $G$ are defined as the closures of $\mathbb C[G]$ in $\mathbb B(\ell^2(G))$ with respect to the operator norm topology and the weak operator topology, respectively \cite{MvN43}.
An important theme in operator algebras is the classification of the group algebras $C_r^*(G)$ and $L(G)$, and hence the understanding of how much information about the group $G$ is encoded in these algebras.

Von Neumann algebras of amenable groups retain only limited information about the groups from which they arise. If $G$ is an infinite abelian group, then $L(G)$ is isomorphic to the abelian algebra $L^\infty(\widehat G, m_{\widehat G})$, where $m_{\widehat G}$ is the Haar measure on $\widehat G$, and hence to
$L^\infty([0,1], \lambda)$, where $\lambda$ is the Lebesgue measure. Thus, $L(G)$ does not have any recollection about torsion freeness or finite generation. Connes' breakthrough result \cite{Co76} asserts that the same strong lack of rigidity occurs for the class of amenable groups $G$ that are icc (i.e., all non-trivial conjugacy classes of $G$ are infinite). More precisely, Connes showed that group von Neumann algebras of icc amenable groups are isomorphic to the hyperfinite II$_1$ factor. 

In sharp contrast, reduced $C^*$-algebras of amenable groups can completely remember the underlying group. A classical result of Scheinberg \cite{Sc74} asserts that all torsion free abelian groups $G$ are {\it $C^*$-superrigid}, i.e. if $C_r^*(G)\cong C_r^*(H)$, then $G\cong H$. Examples of non-abelian $C^*$-superrigid groups have only been discovered relatively recently: certain torsion-free virtually abelian groups in \cite{CKRTW17} and \cite{CLMW25}, two-step nilpotent groups in \cite{ER18}, and free nilpotent groups in \cite{Om18}.

The non-amenable case becomes much more complex in both von Neumann and $C^*$-algebras. Major progress has been made since the advent of Popa’s deformation/rigidity theory \cite{Po07}, leading to the discovery of numerous situations in which algebraic and analytic properties of a group $G$ that can be recovered from its group von Neumann algebra  $L(G)$, see the surveys \cite{Va10a, Io12b, Io17}. Remarkably, Ioana, Popa, and Vaes identified in \cite{IPV10} the first examples of countable groups  $G$ that are  {\it $W^*$-superrigid}, i.e. if $L(G)\cong L(H)$, then $G\cong H$. 
Subsequently, additional examples of $W^*$-superrigid groups have been found \cite{BV12, Be14, CI17, CD-AD20, CD-AD21, CIOS21, CFQOT24, DV24, DV25, AMCFQ26}.

The first non-amenable $C^*$-superrigid groups $G$ are certain amalgamated free products with trivial amenable radical, as shown by Chifan and Ioana \cite{CI17}. Having already established $W^*$-superrigidity for these groups, $C^*$-superrigidity then follows as a consequence of the uniqueness of the trace on $C_r^*(G)$ \cite{BKKO14}. Subsequently, additional examples of $C^*$-superrigid groups arising from amalgamated free products, HNN-extension groups, coinduced groups, and semi-direct products (including wreath products with non-amenable core) have been discovered in \cite{CD-AD20, CD-AD21}. 

Despite these advances, identifying new classes of $C^*$-superrigid groups remains a challenging problem. In this work, we address this question by presenting in Theorem \ref{A} and Corollary \ref{corA} a broad class of $W^*$ and $C^*$-superrigid product groups. 
For presenting the result, we introduce the following class of groups.

\begin{definition}
    A group $G$ belongs to $\mathcal C_{AFP}$ if $G=G_1*_{\Sigma}G_2$ admits an amalgamated free product decomposition satisfying:
\begin{itemize}
    \item [(i)] $\Sigma$ is a common amenable icc subgroup of $G_1$ and $G_2$,

    \item [(ii)] $G_i=G_{i}^1\times G_{i}^2$ is a product of two icc, non-amenable, {\it relatively solid} groups for any $i\in\{1,2\}$.

\end{itemize}

The definition of relatively solid groups is introduced in Section \ref{section.rsg} and examples of such groups include non-amenable hyperbolic groups. Moreover, if the following two conditions hold then we say that  $G$ belongs to $\mathcal C^0_{AFP}$.

\begin{itemize}
    \item [(iii)] For all $i,j\in\{1,2\}$, $G_i^j$ is a generalized wreath product of the form $\mathbb Z_2\wr_{\Gamma/K} \Gamma$, where $K<\Gamma$ is an amenable malnormal subgroup and $\Gamma$ contains an infinite normal property (T) subgroup. 

    \item [(iv)] ${\rm QN}_{G_i}^{(1)}(\Sigma)=\Sigma$ for any $i\in\{1,2\}$ and there exist $g_1,\dots,g_m\in G$ such that $\cap_{k=1}^m g_k \Sigma g_k^{-1}$ is finite.

\end{itemize}

\end{definition}

Note that \cite{CD-AD20} shows that every group $G\in \mathcal C_{AFP}^0$ is both $W^*$ and $C^*$-superrigid. In the next theorem and its corollary, we extend this result to arbitrary finite products of such groups.

\begin{main}\label{A}
Let $G_1,\dots,G_n$ be groups from $\mathcal C^0_{AFP}$ and denote $G=G_1\times\dots\times G_n$. Let $H$ be an arbitrary group and let $\theta: L(G)^t\to L(H)$ be a $*$-isomorphism for some $t>0$. 

Then $t=1$ and $G\cong H$. Moreover, there exist a group isomorphism $\delta: G\to H$, a unitary $w\in L(H)$ and a character $\eta:G\to\mathbb T$ such that $\theta(u_g)= \eta(g) w v_{\delta(g)} w^*$, for any $g\in G$. Here, we denoted by $\{u_g\}_{g\in G}$ and by $\{v_{h}\}_{h\in H}$ the canonical generating unitaries of $L(G)$ and $L(H)$, respectively.
\end{main}

Combining Theorem \ref{A} with the unique trace property of products of groups from $\mathcal C^0_{AFP}$, we obtain the corresponding $C^*$-superrigidity result.

\begin{mcor}\label{corA}
    Let $G_1,\dots,G_n$ be groups from $\mathcal C^0_{AFP}$ and denote $G=G_1\times\dots\times G_n$. Let $H$ be an arbitrary group and let $\theta: C^*_r(G)\to C^*_r(H)$ be a $*$-isomorphism. 

    Then there is a group isomorphism $\delta: G\to H$, a unitary $w\in L(H)$ and a character $\eta:G\to\mathbb T$ such that $\theta(u_g)= \eta(g) w v_{\delta(g)} w^*$, for any $g\in G$. 
\end{mcor}

These examples, along with those obtained independently by Ariza Mej\'{\i}a, Chifan, and Fern\'{a}ndez Quero, are the first $C^*$-superrigid non-amenable product groups \cite{AMCFQ26}. Moreover, since \cite{AMCFQ26} establishes $C^*$-superrigidity for infinite direct sums, they also provide the first examples of non-amenable $C^*$-superrigid groups that are not $W^*$-superrigid.

\begin{example}To provide concrete examples in the class $\mathcal C_{AFP}^0$, let $\Gamma$ be a hyperbolic group admitting an infinite normal subgroup with property (T) and two infinite cyclic malnormal subgroups $K, L < \Gamma$ such that $K \cap g L g^{-1} = \{1\}$ for all $g \in \Gamma$. See \cite[Section~6]{CD-AD20} for a construction of such examples using techniques from geometric group theory \cite{BO06}. The group $G = G_1 *_\Sigma G_2$ belongs to $\mathcal C_{AFP}^0$ whenever $G_1 = G_2 = \mathbb{Z}_2 \wr_{\Gamma/K} \Gamma \times \mathbb{Z}_2 \wr_{\Gamma/K} \Gamma$ and $\Sigma = \{ (g,g) \mid g \in \mathbb{Z}_2^{(\Gamma/K)} \rtimes L \}$. Indeed, note that $\Sigma$ is an icc amenable group, while $G_1$ and $G_2$ are relatively solid; see Proposition~\ref{proposition.rs}.
The condition $K \cap g L g^{-1} = \{1\}$ for all $g \in \Gamma$ ensures that condition (iv) defining the class $\mathcal C_{AFP}^0$ is satisfied.
\end{example} 


The main ingredient in proving Theorem \ref{A} is establishing the following product rigidity for von Neumann algebras of groups in $\mathcal C_{AFP}$ in the spirit of \cite{CdSS15}.


\begin{main}\label{B}
Let $G=G_1\times\dots\times G_n$ be a product of $n\ge 2$ countable groups that belong to $\mathcal C_{AFP}$ and denote $M=L(G)$. Let $H$ be any countable group such that $M^t=L(H)$ for some $t>0.$

Then there exist a product decomposition $H=H_1\times\dots\times H_n$, a unitary $u\in\mathcal U(M^t)$ and some positive numbers $t_1,\dots,t_n$ with $t_1\dots t_n=t$ such that $uL(H_i)u^*=L(G_i)^{t_i}$, for any $i\in \{1,\dots, n\}.$
\end{main}

Additional product rigidity results have been established for products of bi-exact groups, groups with positive first $\ell^2$-Betti numbers, and wreath-like product groups \cite{CU18, CD-AD20, Dr20, CDD23, AMCOS25, DD25, AMCFQ26}. Theorem~\ref{B} further extends this list to include new examples of products of  amalgamated free product groups.

To prove Theorem~\ref{B}, we develop several techniques in the theory of von Neumann algebras. In Section~\ref{section.solidity}, we establish solidity-type results for von Neumann algebras associated with products of groups in the class $\mathcal C_{AFP}$. In Section~\ref{section.product.rigidity}, we show that the von Neumann algebra $L(\tilde G \times \tilde G^0)$ retains the product structure whenever $\tilde G$ is a product of groups from $\mathcal C_{AFP}$ and $\tilde G^0$ is a product of icc relatively solid groups. In Section~\ref{section.peripheral}, we present results that allow us to reconstruct, at the level of von Neumann algebras, the peripheral structure of products of groups belonging to the class $\mathcal C_{AFP}$.

\textbf{Acknowledgments.} JC was supported by EPSRC grant EP/X026647/1. DD was partially supported by the EPSRC grant EP/X026647/1, NSF grant DMS \#2452525, and the Simons Foundation.  For the purpose of Open Access, the authors have applied a CC BY public copyright licence to any Author Accepted Manuscript (AAM) version arising from this submission.

\section{Preliminaries}

\subsection{Terminology}
In this paper we consider {\it tracial von Neumann algebras} $(M,\tau)$, i.e. von Neumann algebras $M$ equipped with a faithful normal tracial state $\tau: M\to\mathbb C.$ This induces a norm on $M$ by the formula $\|x\|_2=\tau(x^*x)^{1/2},$ for all $x\in M$. We will always assume that $M$ is a {\it separable} von Neumann algebra, i.e. the $\|\cdot\|_2$-completion of $M$ denoted by $L^2(M)$ is separable as a Hilbert space.


We denote by $\mathcal Z(M)$ the {\it center} of $M$ and by $\mathcal U(M)$ its {\it unitary group}. For two von Neumann subalgebras $P_1,P_2\subset M$, we denote by $P_1\vee P_2=W^*(P_1\cup P_2)$ the von Neumann algebra generated by $P_1$ and $P_2$. 

Let $P\subset M$ be a unital inclusion of von Neumann algebras. We denote by $E_{P}:M\to P$ the unique $\tau$-preserving {\it conditional expectation} from $M$ onto $P$, by $e_P:L^2(M)\to L^2(P)$ the orthogonal projection onto $L^2(P)$ and by $\langle M,e_P\rangle$ the Jones' basic construction of $P\subset M$. We also denote by $P'\cap M=\{x\in M\mid xy=yx, \text{ for all } y\in P\}$ the {\it relative commutant} of $P$ in $M$ and by $\mathcal N_{M}(P)=\{u\in\mathcal U(M)\mid uPu^*=P\}$ the {\it normalizer} of $P$ in $M$.  We say that $P$ is {\it regular} inside $M$ if the von Neumann algebra generated by $\mathcal N_M(P)$ equals $M$. 

The {\it amplification} of a II$_1$ factor $(M,\tau)$ by a number $t>0$ is defined to be $M^t=p(\mathbb B(\ell^2(\mathbb Z))\bar\otimes M)p$, for a projection $p\in \mathbb B(\ell^2(\mathbb Z))\bar\otimes M$ that satisfies $($Tr$\otimes\tau)(p)=t$. Here Tr denotes the usual trace on $\mathbb B(\ell^2(\mathbb Z))$. Since $M$ is a II$_1$ factor, $M^t$ is well defined. Note that if $M=P_1\bar\otimes P_2$, for some II$_1$ factors $P_1$ and $P_2$, then there exists a natural identification $M=P_1^t\bar\otimes P_2^{1/t}$, for every $t>0.$


Finally, for a positive integer $n$, we denote by $\overline{1,n}$ the set $\{1,\dots, n\}$. For any subset $S\subset \overline{1,n}$ we denote its complement by $\widehat S=\overline{1,n}\setminus S$. If $S=\{i\},$ we will simply write $\widehat i$ instead of $\widehat {\{i\}}$. Also, given
any product group $G=G_1\times \dots\times G_n$, we will denote its subproduct supported on $S$ by $G_S=\times_{i\in S}G_i$.

\subsection {Intertwining-by-bimodules} We next recall from  \cite [Theorem 2.1 and Corollary 2.3]{Po03} the powerful {\it intertwining-by-bimodules} technique of S. Popa.

\begin {theorem}[\!\!\cite{Po03}]\label{corner} Let $(M,\tau)$ be a tracial von Neumann algebra and $P\subset pMp, Q\subset qMq$ be von Neumann subalgebras. Let $\mathcal U\subset\mathcal U(P)$ be a subgroup such that $\mathcal U''=P$.

Then the following are equivalent:

\begin{enumerate}

\item There exist projections $p_0\in P, q_0\in Q$, a $*$-homomorphism $\theta:p_0Pp_0\rightarrow q_0Qq_0$  and a non-zero partial isometry $v\in q_0Mp_0$ such that $\theta(x)v=vx$, for all $x\in p_0Pp_0$.

\item There is no sequence $(u_n)_n\subset\mathcal U$ satisfying $\|E_Q(xu_ny)\|_2\rightarrow 0$, for all $x,y\in M$.

\end{enumerate}
\end{theorem}

If one of the equivalent conditions of Theorem \ref{corner} holds true, we write $P\prec_{M}Q$, and say that {\it a corner of $P$ embeds into $Q$ inside $M$.}

The next lemma is essentially contained in \cite[Section 3]{CD-AD20}. For completeness, we provide all the details.

\begin{lemma}\label{lemma.augm}
Let $H$ be a countable icc group, let $M=L(H)$ and let $\Delta:M\to M\bar\otimes M$ the $*$-homomorphism given by $\Delta(v_h)=v_h\otimes v_h$ for $h\in H$. 
Let $P\subset pMp$ be a von Neumann subalgebra and $\Sigma<H$ be subgroup such that $P\prec_M L(\Sigma)$. 

If $Q\subset qMq$ is a von Neumann subalgebra such that $L(\Sigma)\prec_M Q$, then $\Delta(P)\prec_{M\bar\otimes M} M\bar\otimes Q.$

\end{lemma}

{\it Proof.}
Following the augmentation technique from \cite[Section 3]{CD-AD20}, we consider a Bernoulli action $H\car D$  with abelian base. Denote $\mathcal M=D\rtimes H$ and extend $\Delta$ to the $*$-homomorphism $\Delta:\mathcal M\to \mathcal M\bar\otimes M$  given by $\Delta(du_g)=du_g\otimes u_g$, for all $d\in D$ and $g\in H$. 

Without loss of generality we may assume that $\Sigma$ is infinite. It follows that $P\prec_{\mathcal M} (D\rtimes\Sigma )z$ for any non-zero projection $z\in (D\rtimes\Sigma)'\cap \mathcal M=\mathbb C 1$. By using \cite[Remark 2.2]{DHI16} and \cite[Lemma 2.3]{Dr19b} we further obtain that $\Delta(P)\prec_{\mathcal M \bar\otimes M} (\mathcal M\bar\otimes L(\Sigma))z$ for any non-zero projection $z\in (\mathcal M\bar\otimes L(\Sigma))'\cap \mathcal M \bar\otimes M$. Using $L(\Sigma)\prec_M Q$ and \cite[Lemma 2.4(2)]{Dr19b}, we get $\Delta(P)\prec_{\mathcal M \bar\otimes M}\mathcal M\bar\otimes Q$. Since $\Delta(P)\subset M\bar\otimes M$, it follows that $\Delta(P)\prec_{ M \bar\otimes M} M\bar\otimes Q$.
\hfill$\blacksquare$

\subsection{Relative amenability}
A tracial von Neumann algebra $(M,\tau)$ is called {\it amenable} if there exists a positive linear functional $\Phi \colon \mathbb{B}(L^2(M)) \to \mathbb{C}$ such that $\Phi_{|M} = \tau$ and $\Phi$ is $M$-{\it central}, that is, $\Phi(xT) = \Phi(Tx)$ for all $x \in M$ and $T \in \mathbb{B}(L^2(M))$. 

We next recall the notion of relative amenability introduced by Ozawa and Popa \cite{OP07}. Let $(M,\tau)$ be a tracial von Neumann algebra, let $p \in M$ be a projection, and let $P \subset pMp$ and $Q \subset M$ be von Neumann subalgebras. Following \cite[Definition~2.2]{OP07}, we say that $P$ is {\it amenable relative to $Q$ inside $M$} if there exists a positive linear functional $\Phi \colon p\langle M,e_Q\rangle p \to \mathbb{C}$ such that $\Phi_{|pMp} = \tau$ and $\Phi$ is $P$-central. We say that $P$ is {\it strongly non-amenable relative to $Q$} if $Pp'$ is non-amenable relative to $Q$ for every non-zero projection $p' \in P' \cap pMp$ (equivalently, for every non-zero projection $p' \in \mathcal{N}_M(P)' \cap pMp$ by \cite[Lemma~2.6]{DHI16}).

We continue with the following lemma in the spirit of \cite[Propostion 2.2]{KV15} that explores a connection between intertwining and relative amenability

\begin{lemma}\label{lemma.intertwining}
Let $(M,\tau)$ be a tracial von Neumann algebra. Let $P_0\subset P\subset pMp$ and $Q_0\subset Q\subset qMq$ be von Neumann subalgebras such that there exist projections $p'\in P$ with $p'\in P_0$ or $p'\in P_0'\cap P$, $q_0\in Q$, a non-zero partial isometry $v\in q_0Mp'$ and a $*$-homomorphism $\theta: p'Pp'\to q_0Qq_0$ satisfying $\theta(x)v=vx$ for $x\in p'Pp'$. Assume that the support of $E_{Q}(vv^*)$ equals $q_0$. 
Then the following hold:
\begin{enumerate}
    \item If $\theta(p'P_0p')\prec_{Q} Q_0$, then $P_0\prec_M Q_0$, 

    \item If $\theta(p'P_0p')$ is not strongly non-amenable relative to $Q_0$ inside $Q$, then $P_0$ is not strongly non-amenable relative to $Q_0$ inside $M$.
\end{enumerate}

\end{lemma}

{\it Proof.} (1) 
Since $\theta(p'P_0p')\prec_{Q} Q_0$, there exist projections $p_0\in p'P_0p',q_1\in Q_0$, a non-zero partial isometry $w\in q_1 Q \theta(p_0)$ and a $*$-homomorphism $\Psi: \theta(p_0 P p_0)\to q_1 Q_0 q_1$ such that $\Psi(x)w=wx$, for any $x\in p_0 P_0 p_0$.  Define $\tilde\Psi: p_0 P_0 p_0 \to q_0 Q_0 q_0$ by $\Tilde\Psi(x)=\Psi(\theta(x))$ for any $x\in p_0P_0p_0$ and $\tilde w=wv$.  Note that $\Tilde\Psi(x)\tilde w=\tilde w x$, for any $x\in p_0 P_0 p_0$ and  that $\tilde w\neq 0$.  Otherwise we get $wE_{Q}(vv^*)=0$, which implies that $w=0$, contradiction. By replacing $\tilde w$ by the partial isometry from its polar decomposition, we get that $P_0\prec_M Q_0$.

(2) By \cite[Lemma 2.6(2)]{DHI16} there exists a non-zero projection $q'\in   \mathcal N_{q_0 M q_0}(\theta(p'P_0 p'))'\cap q_0 M q_0$ such that $\theta(p'P_0p')q'$ is amenable relative to $Q_0$ inside $qMq$. Since $vv^*\in \theta(p'P_0 p')'\cap q_0Mq_0$ is a non-zero projection, then
$q'vv^* \in  \theta(p'P_0 p')'\cap q_0Mq_0$ is a non-zero projection as well. Indeed, if $q' vv^*=0$, then $q'E_{Q}(vv^*)=0$, which implies that $q'q_0=0$. This contradicts the fact that $q'\in q_0Qq_0$ is a non-zero projection. Next, note that
\cite[Lemma~2.6]{DHI16} implies that $\theta(p'P_0p')q'vv^*$ is amenable relative to $Q_0$ inside $qMq$, and hence, $q' vP_0v^*q' $ is amenable relative to $Q_0$ inside $qMq$. Let $u\in \mathcal U( M)$ be a unitary for which $v=uv^*v$. We then have that $  P_0v^* v u^*q'u $ is amenable relative to $Q_0$ inside $qMq$. This shows that $P_0$ is not strongly non-amenable relative to $Q_0$ inside $M$.
\hfill$\blacksquare$

\subsection{Quasinormalizers of groups and von Neumann algebras}\label{subsection.quasi.normalizer} 

\noindent Given a group inclusion $H<G$, the one-sided quasi-normalizer ${\rm QN}^{(1)}_G(H)$ is the semigroup of all $g\in G$ for which there exists a finite set $F\subset G$ such that $Hg\subset FH$ \cite[Section 5]{FGS10}; equivalently, $g\in {\rm QN}^{(1)}_G(H)$ if and only if $[H: gHg^{-1}\cap H]<\infty$. The quasi-normalizer ${\rm QN}_G(H)$ is the group of all $g\in G$ for which exists a finite set $F\subset G$ such that $Hg\subset FH$ and $gH\subset HF$.

\noindent Given an inclusion $N \subseteq M$ of tracial von Neumann algebras we define the quasi-normalizer  $\mathcal {QN}_{\mathcal M}(\mathcal N)$ as the set of all elements $x\in M$ for which there exist $x_1,...,x_n\in M$ such that $N x\subseteq \sum x_i N$ and $x N \subseteq \sum N x_i$ (see \cite[Definition 4.8]{Po99}). Also the one-sided quasi-normalizer $\mathcal {QN}^{(1)}_{\mathcal M}(\mathcal N)$ is defined as the set of all elements $x\in M$ for which there exist $x_1,...,x_n\in M$ such that $N x\subseteq \sum x_i N$ \cite{FGS10}.

\noindent We record now some formulas for the quasi-normalizer of corners.

\begin{lemma}[\!\!{\cite{Po03,FGS10}}]\label{lemma.qn.algebras}
Let $P\subset M$ be tracial von Neumann algebras. For any projection $p\in P$, the following hold:
\begin{itemize}
    \item $W^*({\rm \mathcal{QN}^{(1)}_{pM p}}(pP p))=pW^*({\rm \mathcal{QN}^{(1)}_{M}}(P))p$,
    
    \item $W^*({\rm \mathcal{QN}_{pM p}}(pP p))=pW^*({\rm \mathcal{QN}_{M}}(P))p$. 
\end{itemize}

Moreover, for any projection $p'\in P'\cap M$, the following hold:
\begin{itemize}
    \item $W^*({\rm \mathcal{QN}^{(1)}_{p'M p'}}(P p'))=p'W^*({\rm \mathcal{QN}^{(1)}_{M}}(P))p'$,

    \item $W^*({\rm \mathcal{QN}_{p'M p'}}(P p'))=p'W^*({\rm \mathcal{QN}_{M}}(P))p'$.
\end{itemize}
\end{lemma}

\noindent The following result provides a relation between the group theoretical quasi-normalizer and the von Neumann algebraic one.

    \begin{lemma}[\!\!{\cite[Corollary 5.2]{FGS10}}]\label{lemma.qn.groups}
Let $H<G$ be countable groups. Then the following hold:
\begin{enumerate}
    \item $W^*(\mathscr{QN}^{(1)}_{L(G)}(L(H)))=L(K)$,  where $K<G$ is the subgroup generated by ${\rm QN}^{(1)}_G(H)$. In particular, if ${\rm QN}^{(1)}_G(H)=H$, then $\mathscr{QN}^{(1)}_{L(G)}(L(H))=L(H)$.
    
    \item $W^*(\mathscr{QN}_{L(G)}(L(H)))=L({\rm QN}_G(H))$.
\end{enumerate}

\end{lemma}

The following lemma essentially contained in the proof of \cite[Theorem 5.1]{IPP05}. We add details for the convenience of the reader.

\begin{lemma}[\!\!\cite{IPP05}]\label{lemma.normalizer}
Let $G=G_1*_\Sigma G_2$ be an amalgamated free product group and let $M=L(G)$. Let $P\subset pL(G)p$ be a von Neumann subalgebra such that $P\prec_{L(G)} L(G_1)$ and $P\nprec_{L(G)} L(\Sigma)$.  Denote $Q=\mathcal {QN}_{L(G)}^{(1)}(P)''$.

Then $Q\prec_{L(G)} L(G_1)$. Moreover, if $G_1$ is icc and  $Qz\prec^s_{L(G)} L(G_1)$ for some non-zero projection $z\in \mathcal Z(Q)$, then there is a unitary $u\in \mathcal U(L(G))$ such that $u Qz u^*\subset L(G_1)$.
\end{lemma}

{\it Proof.}
Since $P\prec_{L(G)} L(G_1)$, there exist projections $p_0\in P, q_0\in L(G_1)$, a non-zero partial isometry $v\in q_0 L(G) p_0$ and a $*$-isomorphism $\theta: p_0 P p_0 \to \mathcal C\subset q_0L(G_1)q_0$ such that $\theta(x)v=vx$, for any $x\in p_0 P p_0$. Thus, $\mathcal C vv^* = vp_0Pp_0v^*$. From assumption and Lemma \ref{lemma.intertwining} we get $\mathcal C\nprec_{L(G_1)} L(\Sigma)$, and hence by \cite[Corollary 2.10]{CI17} (see also \cite[Theorem 1.1]{IPP05}) it follows that $vv^*\in L(G_1)$. By letting $u\in L(G)$ be a unitary extending $v$, we deduce that $up_0Pp_0v^*vu^*\subset L(G_1)$. By assumption $up_0Pp_0v^*vu^*\nprec_{L(G_1)}L(\Sigma)$, and thus by \cite[Corollary 2.10]{CI17} and Lemma \ref{lemma.qn.algebras}, it follows that 
$$
u v^*v \mathcal{QN}^{(1)}_{L(G)}(P)'' v^*v u^* \subset L(G_1).
$$
This implies that $Q\prec_M L(G_1)$. 

For the moreover part, note that $L(G_1)$ is a II$_1$ factor. Thus, there exists a maximal central projection $z_1\in \mathcal Z(Q)$ with $z_1\leq z$ for which there is a unitary $u_1\in \mathcal U(L(G))$ such that $u_1 Qz_1 u_1^*\subset L(G_1)$. Our goal is to prove that $z_1=z$. 

Assume by contradiction that $z_0=z-z_1$ is a non-zero projection. Since $Qz_0\prec_{L(G)} L(G_1)$, by proceeding as in the first part of the proof, we get a unitary $u_2\in \mathcal U(L(G_1))$ a non-zero central projection $z_2\in \mathcal Z(Q)$ with $z_2\leq z_0$ such that $u_2 Qz_2 u_2^*\subset L(G_1)$. 
Define $\tilde z=z_1+z_2$.
Since $L(G_1)$ is a II$_1$ factor, we modify $u_2$ so that we may assume that $u_1z_1u_1^*$ and $u_2z_2u_2^*$ are orthogonal.
Define the partial isometry $\tilde u=u_1z_1+u_2z_2$ and note that $\tilde u^* \tilde u = z_1+z_2$ and $\tilde u Q \tilde z \tilde u^*\subset L(G_1)$. By extending $\tilde u$ to a unitary $u$, we get that $u Q\tilde z u^*\subset L(G_1)$. This contradicts the maximality of $z_1$, ending the proof.
\hfill$\blacksquare$

\subsection{Relatively solid groups}\label{section.rsg}

The main result of \cite{Oz03} (see also \cite{CS11}) shows that the von Neumann algebra $L(\Gamma)$ of a non-amenable hyperbolic group is {\it solid}, i.e. if $A\subset L(\Gamma)$ is a diffuse subalgebra, then the relative commutant $A'\cap L(\Gamma)$ is amenable. The work of \cite{PV12} (see \cite[Lemma 5.2]{KV15}) straightens this result, by showing that any non-amenable hyperbolic group is {\it relatively solid}.

\begin{definition}
A countable group $\Gamma$ is called {\it relatively solid} if for any trace preserving action $\Gamma\car N$ and any two commuting von Neumann subalgebras $Q_1,Q_1\subset q(N\rtimes\Gamma)q$, we have that $Q_1\prec_{N\rtimes\Gamma}N$ or $Q_2$ is amenable relative to $N$ inside $N\rtimes\Gamma$.   
\end{definition}

\begin{lemma}\label{lemma.rs}
Let $\Gamma$ be a relatively solid group and $\Gamma\car N_0$ an trace presering action. Let $N$ be a tracial von Neumann algebra for which $N_0\rtimes\Gamma\subset pNp$. 

Let $Q_1,Q_2\subset qNq$ be commuting von Neumann subalgebras such that $Q_1$ or $Q_2$ contains a II$_1$ subfactor. If $Q_1\vee Q_2\prec_{N}N_0\rtimes\Gamma$ and $Q_2$ is strongly non-amenable relative to $N_0$ inside $N$, then  $Q_1\prec_N N_0$.
\end{lemma}

{\it Proof.} Assume $Q_1$ contains a II$_1$ subfactor (we may proceed similarly for $Q_2$).
By assumption there are projections $p\in Q_1\vee Q_2$ and $r\in N_0\rtimes\Gamma$, a non-zero partial isometry $v\in rNp$ and a $*$-homomorphism $\theta: q_0(Q_1\vee Q_2)q_0\to r(N_0\rtimes\Gamma)r$ such that $\theta(x)v=vx$ for any $x\in q_0(Q_1\vee Q_2)q_0$. Since $Q_1$ contains  a II$_1$ subfactor, we may assume by \cite[Lemma 4.5]{CdSS17} that $q_0\in Q_1$. By Lemma \ref{lemma.intertwining}, $\theta(Q_2q_0)$ is non-amenable relative to $N_0$ inside $N_0\rtimes\Gamma$. Since $\Gamma_0$ is relatively solid, it follows that $\theta(q_0Q_1q_0)\prec_{N_0\rtimes\Gamma}N_0$, which implies by Lemma \ref{lemma.intertwining} that $Q_1\prec_N N_0$.
\hfill$\blacksquare$

The following result shows that the notion of relative solidity is closed by considering wreath products with amenable base. 
The proof is standard and we include a proof for the convenience of the reader.

\begin{proposition}\label{proposition.rs}
Let $G=A\rtimes_{\Gamma/K}\Gamma$ be a generalized wreath product group with $A$ an amenable group and $K<\Gamma$ an amenable malnormal subgroup. If $\Gamma$ is relatively solid, then $G$ is relatively solid.
\end{proposition}

{\it Proof.}
Let $G\car N$
 be a trace preserving action and denote $M=N\rtimes G$. Let $Q_1,Q_2\subset qMq$ be commuting von Neumann subalgebras with $Q_2$ non-amenable relative to $N$ inside $M$. By \cite[Lemma 2.6]{DHI16}, take a non-zero projection $z\in (Q_1\vee Q_2)'\cap qNq$ such that $Q_2$ is strongly non-amenable relative to $N$ inside $M$. By  \cite[Theorem 4.2]{IPV10} and \cite[Corollary 4.3]{IPV10}, it follows that (i) $Q_1z\prec_M N$, or (ii) $(Q_1\vee Q_2)z\prec_M N\rtimes (A^{(\Gamma/K)}\rtimes K) $, or (iii) $(Q_1\vee Q_2)z\prec_M N\rtimes\Gamma.$

If (i) holds, we are done. If (ii) holds, since $K$ and $A$ are amenable groups, \cite[Lemma 2.4]{DHI16} and \cite[Lemma 2.6]{DHI16}
 imply that $Q_2z$ is not strongly non-amenable relative to $N$, contradiction. In the case (iii), Lemma \ref{lemma.rs} implies that $Q_1z\prec_{M}N$.
\hfill$\blacksquare$

\subsection{Amalgamated free product groups and class $\mathcal C_{AFP}$}

The following theorem is a direct consequence of \cite{Io12a,Va13} and  its proof proceeds along the same lines as the proofs of \cite[Lemma 5.2]{KV15} and \cite[Theorem 3.1]{CdSS17}.

\begin{theorem}\label{theorem.classification}
Let $M=M_1*_B M_2$ be an amalgamated free product of tracial von Neumann algebras. If $Q_1,Q_2\subset pMp$ are commuting von Neumann subalgebras, then one of the following holds:
\begin{itemize}
    \item $Q_i\prec_M B$ for some $i\in\{1,2\}$.

    \item $Q_1\vee Q_2\prec_M M_i$ for some $i\in\{1,2\}$.

    \item $Q_1\vee Q_2$ is amenable relative to $B$ inside $M$.
\end{itemize}
\end{theorem}

{\it Proof.} Assume that $Q_1\nprec_M B$. By \cite[Corollary F.14]{BO08}, there exists a diffuse abelian von Neumann subalgebra  $Q_0\subset Q_1$ satisfying $Q_0\nprec_M B$. Since $Q_0$ is amenable, by \cite[Thereom A]{Va13}, we get that (i) $Q_2\prec_{M} M_i$ for some $i\in \{1,2\}$ or (ii) $Q_2$ is amenable relative to $B$ inside $M$. If (i) holds, then Lemma \ref{lemma.normalizer} implies that $Q_2\prec_M B$ or $Q_1\vee Q_2 \prec_M M_i$. If (ii) holds, then \cite[Theorem A]{Va13} implies that $Q_2\prec_M B$ or $Q_1 \vee Q_2\prec_{M} M_j$ for some $j\in\{1,2\}$ or $Q_1\vee Q_2$ is amenable relative to $B$ inside $M$.  In all cases, the conclusion follows.
\hfill$\blacksquare$


\begin{corollary}\label{corollary.classification}
Let $M=M_1*_B M_2$ be an amalgamated free product of tracial von Neumann algebras and let $N$ be a tracial von Neumann algebra satisfying $M\subset pNp$.

Let $Q_1,Q_2\subset qNq$ be commuting von Neumann subalgebras such that $Q_1$ or $Q_2$ contains a II$_1$ subfactor. If $Q_1\vee Q_2\prec_{N} M$ and $Q_2$ is strongly non-amenable relative to $B$ inside $N$, then either $Q_1\prec_N B$ or $Q_1\vee Q_2 \prec_N M_i$ for some $i\in\{1,2\}$.
\end{corollary}

{\it Proof.} Assume $Q_1$ contains a II$_1$ subfactor (we may proceed similarly for $Q_2$).
By assumption, there are projections $q_0\in Q_1\vee Q_2$ and $r\in M$, a non-zero partial isometry $v\in rNp$ and a $*$-homomorphism $\theta: q_0(Q_1\vee Q_2)q_0\to rMr$ such that $\theta(x)v=vx$ for any $x\in q_0(Q_1\vee Q_2)q_0$. Since $Q_1$ contains a II$_1$ subfactor, we may assume by \cite[Lemma 4.5]{CdSS17} that $q_0\in Q_1$. By Theorem \ref{theorem.classification}, get that one of the following holds: 
\begin{enumerate}
    \item $\theta(q_0Q_iq_0)\prec_M B$ for some $i\in\{1,2\}$,

    \item $\theta(q_0(Q_1\vee Q_2)q_0)\prec_M M_i$ for some $i\in\{1,2\}$,

    \item $\theta(q_0(Q_1\vee Q_2)q_0)$ is amenable relative to $B$ inside $M$.
\end{enumerate}

If (1) holds with $i=1$, Lemma \ref{lemma.intertwining} gives that $Q_1\prec_{N}B$. If (1) holds with $i=2$, the same lemma gives that $Q_2\prec_N B$. Using \cite[Lemma 2.4 and Lemma 2.6]{DHI16} we contradict that $Q_2$ is strongly non-amenable relative to $B$ inside $N$. If (2) holds, Lemma \ref{lemma.intertwining} gives $Q_1\vee Q_2\prec_{N} M_i$. Finally, if (3) holds, then $\theta(Q_2 q_0)$ is amenable relative to $B$ inside $M$, and by Lemma \ref{lemma.intertwining} we get that $Q_2$ is not strongly non-amenable relative to $B$ inside $M$, contradiction.
\hfill$\blacksquare$





We end the section by showing that products of groups from $\mathcal C_{AFP}^0$ have trivial amenable radical.



\begin{lemma}\label{lemma.trivial.radical}
Let $G_1,\dots,G_n$ be groups from class $\mathcal C_{AFP}^0$. Then $G_1\times\dots\times G_n$ has trivial amenable radical. 
\end{lemma}

{\it Proof.} 
Denote $G= G_1\times\dots\times G_n$ and $M=L(G)$. For any $i\in \{1,\dots,n\}$, write $G_i=G_{i,1}*_{\Sigma_i} G_{i,2}$. Note that we can take $\Sigma_i$ such that there exist group elements $g_i^1, \dots, g_i^{m_i}\in G_i$ for which $\cap_{k=1}^{m_i} g_i^k \Sigma_i (g_i^k)^{-1}$ is finite.
Let $A$ be a normal amenable subgroup of $G$. Since $L(A)\subset M$ is regular, \cite[Theorem A]{Va13} implies that $L(A) \prec_{M}^s L(\Sigma_1\times \dots\times\Sigma_n)$. By \cite[Proposition 8]{HPV11}, we derive that $L(A) \prec_{M}^s L((\cap_{k=1}^{m_1} g_1^k \Sigma_1 (g_1^k)^{-1})\times \dots\times (\cap_{k=1}^{m_n} g_n^k \Sigma_n (g_n^k)^{-1}))$, implying that $A$ has to be a finite group. Since $G$ is icc and $A$ is normal in $G$, it follows that $A$ is trivial.
\hfill$\blacksquare$

\section{Solidity type results}\label{section.solidity}

The goal of this section is to provide two solidity-type results for
group von Neumann algebras $L(\tilde G\times\tilde G^0)$ where $\tilde G$ is a product of groups from $\mathcal C_{AFP}$ and $\tilde G^0$ is a product of icc relatively solid groups. 

We first establish the following notation that will be assumed throughout the section.

\begin{notation}\label{notation} Let $\tilde G=G_1\times\dots\times G_n$ be a product of $n\ge 1$ groups that belong to $\mathcal C_{AFP}$. For any $1\leq j\leq n$ write $G_j=G_{j,1}*_{\Sigma_j} G_{j,2}$ where:
\begin{itemize}
    \item $\Sigma_j$ is an icc amenable group,

    \item $G_{j,k}=G_{j,k}^{1} \times G_{j,k}^{2}$ is a product of icc, relatively solid groups, for any $1\leq k\leq 2.$

\end{itemize}

Let $\tilde G^0=G_1^0\times\dots\times G_m^0$ is a product of $m\ge 0$ icc relatively solid groups. Let $G=\tilde G\times \tilde G^0$ and denote $M=L(G)$.

Let $D\subset N$ be an inclusion of tracial von Neumann algebras such that $D\bar\otimes M^t\subset N$ for some $t>0$.
Let $H$ be a countable group such that $N=L(H)$. Following \cite{PV09}, we let $\Delta:N\to N\bar\otimes N$ be the $*$-homomorphism given by $\Delta(v_h)=v_h\otimes v_h$, for any $h\in H$.

\end{notation}

\begin{lemma}\label{lemma.classification}
Let $P_0, P_1$ be commuting von Neumann subalgebras of $N$ such that $P_0\vee P_1\prec_{N} D\bar\otimes M^t$.

If $P_1$ is  strongly non-amenable relative to $D$ inside $N$ and $P_1$ contains a II$_1$ subfactor, then one of the following holds:
\begin{enumerate}
    \item $P_0\prec_{N}D\bar\otimes L(\tilde G_{\widehat j}\times \Sigma_j)^t \bar\otimes L(\tilde G^0)$ for some $j\in \{1,\dots,n\}$, or

    \item $P_0\vee P_1\prec_{N}D\bar\otimes L(\tilde G_{\widehat j}\times \tilde G_{j,k})^t \bar\otimes L(\tilde G^0)$ for some $j\in \{1,\dots,n\}$ and $k\in \{1,2\}$, or

    \item $P_0\prec_{N}D\bar\otimes L(\tilde G)^t \bar\otimes L(\tilde G^0_{\widehat i})$ for some $i\in \{1,\dots,m\}$.
\end{enumerate}
\end{lemma}

{\it Proof.} For ease of notation, we will assume that $t=1$.
By assumption there are projections $p\in P_0\vee P_1$ and $r\in D\bar\otimes M$, a non-zero partial isometry $v\in rNp$ and a $*$-homomorphism $\theta: p(P_0\vee P_1)p\to r(D\bar\otimes M)r$ such that $\theta(x)v=vx$ for any $x\in p(P_0\vee P_1)p$. Since $P_1$ contains a  II$_1$ subfactor, we may assume by \cite[Lemma 4.5]{CdSS17} that $p\in P_1$. By Lemma \ref{lemma.intertwining}, $\theta(pP_1p)$ is non-amenable relative to $D$ inside $D\bar\otimes M$. By \cite[Lemma 2.8]{DHI16} one of the following holds:
\begin{itemize}
    \item [(a)] $\theta(pP_1p)$ is non-amenable relative to $D\bar\otimes L(\tilde G_{\widehat j})\bar\otimes L(\tilde G^0)$ inside $D\bar\otimes M$ for some $j\in\{1,\dots,n\}$, or

    \item [(b)]  $\theta(pP_1p)$ is non-amenable relative to $D\bar\otimes L(\tilde G)\bar\otimes L(\tilde G_{\widehat i}^0)$ inside $D\bar\otimes M$ for some $i\in\{1,\dots,m\}$.
\end{itemize}

By \cite[Lemma 2.6]{DHI16} there is a non-zero projection $e\in \theta(p(P_0\vee P_1)p)'\cap r(D\bar\otimes M)r$ such that one of the following holds:
\begin{itemize}
    \item [(a.0)] $\theta(pP_1p)e$ is strongly non-amenable relative to $D\bar\otimes L(\tilde G_{\widehat j})\bar\otimes L(\tilde G^0)$ inside $D\bar\otimes M$, or

    \item [(b.0)] $\theta(pP_1p)e$ is strongly non-amenable relative to $D\bar\otimes L(\tilde G)\bar\otimes L(\tilde G_{\widehat i}^0)$ inside $D\bar\otimes M$.
\end{itemize}

By Corollary \ref{corollary.classification} and Lemma \ref{lemma.rs}, we get that one of the following holds:
\begin{enumerate}
    \item [(i)] $\theta(P_0p)\prec_{D\bar\otimes M}D\bar\otimes L(\tilde G_{\widehat j}\times \Sigma_j) \bar\otimes L(\tilde G^0)$ for some $j\in \{1,\dots,n\}$, or

    \item [(ii)] $\theta(p(P_0\vee P_1)p)\prec_{D\bar\otimes M}D\bar\otimes L(\tilde G_{\widehat j}\times \tilde G_{j,k}) \bar\otimes L(\tilde G^0)$ for some $j\in \{1,\dots,n\}$ and $k\in \{1,2\}$, or

    \item [(iii)] $\theta(P_0p)\prec_{D\bar\otimes M}D\bar\otimes L(\tilde G) \bar\otimes L(\tilde G^0_{\widehat i})$ for some $i\in \{1,\dots,m\}$.
\end{enumerate}
The conclusion now follows by Lemma \ref{lemma.intertwining}.
\hfill$\blacksquare$

The following result shows, in particular, that since $\tilde G$ is a product of $n$ groups from $\mathcal C_{AFP}$, then $L(\tilde G)$ cannot contain more than $2n$ commuting non-amenable subfactors.

\begin{lemma}\label{lemma.count}
Let $s\ge 1$ be an integer and $P_0, P_1,\dots,P_s$ be commuting von Neumann subalgebras of $N$ such that $P_0\vee P_1\vee\dots \vee P_s \prec_{N} D\bar\otimes M^t$.

If $P_i$ is a factor that is strongly non-amenable relative to $D$ inside $N$ for any $1\leq i\leq n$, then $|s|\leq 2n+m$. 

Moreover, if equality holds, then $P_0\prec_{N} D$.
\end{lemma}

{\it Proof.} Once again, without any loss of generality, we may assume that $t=1$. The proof follows by induction over $s$.
By assumption there exist projections $p\in P_0\vee P_1\vee \dots\vee P_s$ and $r\in D\bar\otimes M$, a non-zero partial isometry $v\in rNp$ and a $*$-homomorphism $\theta: p(P_0\vee P_1\vee \dots\vee P_s)p\to r(D\bar\otimes M)r$ such that $\theta(x)v=vx$ for any $x\in p(P_0\vee P_1\vee \dots\vee P_s)p$. Since $P_s$ is a II$_1$ factor, we may assume by \cite[Lemma 4.5]{CdSS17} that $p\in P_s$. By Lemma \ref{lemma.classification}, one of the following holds:
\begin{itemize}
    \item [(i)] $P_0\vee P_1\vee \dots\vee P_{s-1} \prec_{N} D\bar\otimes L(\tilde G_{\widehat j} \times \Sigma_j)\bar\otimes L(\tilde G^0) $ for some $j\in\{1,\dots,n\}$, or

    \item [(ii)] $P_0\vee P_1\vee \dots\vee P_{s} \prec_{N} D\bar\otimes L(\tilde G_{\widehat j} \times G_{j,k})\bar\otimes L(\tilde G^0)$ for some $j\in\{1,\dots,n\}$ and $k\in\{1,2\}$, or

    \item [(iii)] $P_0\vee P_1\vee \dots\vee P_{s-1}\prec_{N} D\bar\otimes L(\tilde G) \bar\otimes L(\tilde G^0_{\widehat i})$ for some $i\in \{1,\dots,m\}$.
\end{itemize}

In the case (i), by induction we get that $s-1\leq 2(n-1)+m$, which proves the lemma. In the case (ii), the conclusion follows by repeating the same arguments finitely many times since we will reduce to the case in which $n=0$ and this can be settled by a separate inductive argument.
In the case (iii), by induction we get $s-1\leq 2n+(m-1)$, which proves the first part of the lemma.

 To prove the moreover part, assume now that $s=2n+m$. By using the first part of the proof together with multiple applications of Lemma \ref{lemma.rs} and Corollary \ref{corollary.classification}, we must get that $P_0 \prec_{N} D$ since the case (a.1) above can never happen.
\hfill$\blacksquare$

We continue with the following lemma, which is instrumental in applying the ultrapower technique from \cite{Io11} to obtain non-amenable commuting subgroups in the mysterious group $H$.

\begin{lemma}\label{lemma.ultrapower} Assume $m=0$ and $M=N$.
For all $i\in \{1,\dots,n\}$ and $ j,k \in \{1,2\}$, there exist $a\in \{1,\dots,n\}$ and $ b,c \in \{1,2\}$ such that 
$$
\Delta(L(G_{i,j}^k\times G_{\widehat i})^t)\prec_{M^t\bar\otimes M^t} M^t\bar\otimes L(G_{a,b}^c\times G_{\widehat a})^t.
$$

\end{lemma}

{\it Proof.} For ease of notation, we assume $t=1$. Without loss of generality, we may assume that $i=j=k=1$.
Throughout the proof we use \cite[Proposition 7.2]{IPV10} which states that if $P\subset M$ is a non-amenable subfactor, then $\Delta(P)$ is strongly non-amenable relative to $M\otimes 1$ inside $M\bar\otimes M$. 

Since $G_{1,1}^2$ is a non-amenable icc group, Corollary \ref{corollary.classification} together with \cite[Lemma 2.8]{DHI16} imply that there exist $a\in\{1,\dots,n\}$ and $b\in \{1,2\}$ such that
$$
\Delta(L(G_{1,1}^1\times G_{\widehat 1}))\prec_{M\bar\otimes M} M\bar\otimes L(G_{a,b}\times G_{\widehat a}).
$$
By Lemma \ref{lemma.normalizer}, we get that one of the following holds:
\begin{enumerate}
    \item $
\Delta(L(G_{1,1}\times G_{\widehat 1}))\prec_{M\bar\otimes M} M\bar\otimes L(G_{a,b}\times G_{\widehat a})
$, or

    \item $\Delta(L(G_{1,1}^1\times G_{\widehat 1}))\prec_{M\bar\otimes M} M\bar\otimes L(\Sigma_{a}\times G_{\widehat a})$.
\end{enumerate}

The second option is impossible by Lemma \ref{lemma.count} since we would get $2(n-1)+1\leq 2(n-1)$. Assume now that the first option holds. By Lemma \ref{lemma.classification}, we get that one of the following holds:

\begin{enumerate}
    \item [(i)]
$
 \Delta(G_{1,1}^1 \times G_{\widehat 1}))\prec_{M\bar\otimes M} M\bar\otimes L(G_{a,b}^{\alpha} \times G_{\widehat a})$, or

    \item [(ii)] $
\Delta(G_{1,1}^1 \times G_{\widehat 1}))\prec_{M\bar\otimes M} M\bar\otimes L(G_{a,b}\times G_{\widehat {\{a,a'\}}}\times \Sigma_{a'})$, or

    \item [(iii)] $\Delta(G_{1,1} \times G_{\widehat 1}))\prec_{M\bar\otimes M} M\bar\otimes L(G_{a,b}\times G_{\widehat {\{a,a'\}}}\times G_{a', k})$.
\end{enumerate}

Note that if $n=1$, then (i) must hold.

In the case (i), we are done. In the case (ii), Lemma \ref{lemma.count} gives that $1+2(n-1)\leq 2+2(n-2)$, contradiction.  In the case (iii), we note that Lemma \ref{lemma.normalizer} and \cite[Proposition 7.2]{IPV10} imply that $\Delta(G_{\widehat 1}))\prec_{M\bar\otimes M} M\bar\otimes L(\Sigma_1\times G_{\widehat {\{a,a'\}}}\times \Sigma_{a'})$. By Lemma \ref{lemma.count}, we get $2(n-1)\leq 2(n-2)$, which again is a contradiction. 
\hfill$\blacksquare$

\section{Product rigidity for groups in $\mathcal C_{AFP}$ and relatively solid groups}\label{section.product.rigidity}

The goal of this section is to prove the following product rigidity result, in the sense of \cite{CdSS15}, which will be used in the proof of Theorem~\ref{B}.
 We are showing that the von Neumann algebra $L(\tilde G\times\tilde G^0)$ remembers the product structure whenever $\tilde G$ is a product of groups from $\mathcal C_{AFP}$ and $\tilde G^0$ is a product of icc relatively solid groups.

\begin{theorem}\label{theorem.product.rigidity}
Let  $G=\tilde G\times\tilde G^0$, where $\tilde G=G_1\times\dots\times G_n$ is a product of $n\ge 1$ countable groups that belong to $\mathcal C_{AFP}$ and $\tilde G^0=G_1^0\times\dots\times G_m^0$ is a product of $m\ge 1$ icc relatively solid groups.

Denote $M=L(G)$ and let $H$ be any countable group such that $M^t=L(H)$ for some $t>0$.

Then there exist a product decomposition $H=\tilde H\times\tilde H^0$, a unitary $u\in\mathcal U(M)$ and positive numbers $t_1,t_2$ with $t_1t_2=t$ such that $uL(\tilde H)u^*=L(\tilde G)^{t_1}$ and $uL(\tilde H^0)u^*=L(\tilde G^0)^{t_2}$.

\end{theorem}

{\it Proof.} For ease of notation, we only consider the case when $t=1$. Let $\Delta:M\to M\bar\otimes M$ the $*$-homomorphism given by $\Delta(v_h)=v_h\otimes v_h$ for $h\in H$.  The proof is divided between three claims.

{\bf Claim 1.} For any $i\in \{1,\dots,m\}$, the following hold:
\begin{itemize}
    \item [(a)] $\Delta(L(\tilde G\times \tilde G^0_{\widehat i}))\prec M\bar\otimes L(\tilde G\times \tilde G^0_{\widehat i})$,

    \item [(b)] There is a subgroup $\Theta_{\widehat i}<H$ satisfying
$ L(\tilde G\times \tilde G^0_{\widehat i})\prec^s_M L(\Theta_{\widehat i})\text{ and } L(\Theta_{\widehat i})\prec_{M}^s L(\tilde G\times \tilde G^0_{\widehat i})$.
\end{itemize}

{\it Proof of Claim 1.} Let $i\in\{1,\dots,m\}$. Since $G_i^0$ is icc non-amenable, Lemma \ref{lemma.classification} implies that one of the following holds:
\begin{enumerate}
    \item [(i)] $\Delta(L(\tilde G\times \tilde G^0_{\widehat i}))\prec M\bar\otimes L(\Sigma_j \times \tilde G_{\widehat j} \times \tilde G^0)$ for some $j\in\{1,\dots,n\}$, or

    \item [(ii)] $\Delta(L(G))\prec M\bar\otimes L(G_{j,k} \times \tilde G_{\widehat j} \times \tilde G^0)$ for some $j\in\{1,\dots,n\}$ and $k\in\{1,2\}$, or

    \item [(iii)] $\Delta(L(\tilde G\times \tilde G^0_{\widehat i}))\prec M\bar\otimes L(\tilde G\times \tilde G^0_{\widehat {i'}})$ for some $i'\in \{1,\dots,m\}$.
\end{enumerate}

If (i) holds, then Lemma \ref{lemma.count} implies that $2n+m-1\leq 2(n-1)+m$, which is impossible. If (ii) holds, \cite[Lemma 7.2]{IPV10} implies that $[G_j:G_{j,k}]<\infty$, contradiction.
Thus, (iii) holds. By \cite[Theorem 4.2]{Dr20} there exists a subgroup $\Theta_{\widehat i}<H$ with non-amenable centralizer $C_H(\Theta_{\widehat i})$ such that
\begin{equation}\label{eq.double}
    L(\tilde G\times \tilde G^0_{\widehat i}) \prec_M L(\Theta_{\widehat i}) \text{ and } L(\Theta_{\widehat i})\prec_M L(\tilde G\times \tilde G^0_{\widehat {i'}}).
\end{equation}
We continue by showing that $ L(\Theta_{\widehat i})\prec^s_M L(\tilde G\times \tilde G^0_{\widehat {i'}})$. Let $z\in L(\Theta_{\widehat i}C_{H}(\Theta_{\widehat i}))'\cap M$ be a non-zero projection. Since $C_{H}(\Theta_{\widehat i})$ is non-amenable (see also \cite[Lemma 2.13]{CD-AD20}), we may apply Lemma \ref{lemma.classification} and get that  one of the following holds:
\begin{enumerate}
    \item [(iii.1)] $ L(\Theta_{\widehat i})z\prec_M L(G_{j,k}\times \tilde G_{\widehat j}\times \tilde G^0)$ for some $j \in \{1,\dots,n\}$ and $k\in\{1,2\}$, or

    \item [(iii.2)] $ L(\Theta_{\widehat i})z\prec_M L(\tilde G\times \tilde G^0_{\widehat {i''}})$ for some $i''\in \{1,\dots,m\}$.
\end{enumerate}

First, we note that (iii.1) cannot hold. Otherwise, by using Lemma \ref{lemma.augm} we get $\Delta(L(\tilde G\times \tilde G^0_{\widehat i}) )\prec_{M\bar\otimes M} M\bar\otimes L(G_{j,k}\times \tilde G_{\widehat j}\times \tilde G^0)$. By Lemma \ref{lemma.normalizer} and \cite[Lemma 7.2]{IPV10}, we get that 
$\Delta(L(\tilde G\times \tilde G^0_{\widehat i}) )\prec_{M\bar\otimes M} M\bar\otimes L(\Sigma_j\times \tilde G_{\widehat j}\times \tilde G^0)$. By Lemma \ref{lemma.count}, we deduce that $2n+m-1\leq 2(n-1)+m$, contradiction.

Hence, (iii.2) holds. Note that $i'=i''$. Otherwise, by Lemma \ref{lemma.augm} we get that
\begin{equation}\label{eq.2}
\Delta(L(\tilde G\times \tilde G^0_{\widehat i}))\prec M\bar\otimes L(\tilde G\times \tilde G^0_{\widehat {i''}})   \,. 
\end{equation}
Note that $\mathcal N_{M\bar\otimes M}(\Delta (L(\tilde G\times \tilde G^0_{\widehat i})))'\cap M\bar\otimes M=\mathbb C 1$.
By combining the intertwinings from (iii) and \eqref{eq.2} and using \cite[Lemma~2.4]{DHI16}, we get 
$
\Delta(L(\tilde G\times \tilde G^0_{\widehat i}))\prec M\bar\otimes L(\tilde G\times \tilde G^0_{\widehat {\{i',i''\}}}).
$
By applying Lemma \ref{lemma.count}, we get that $2n+m-1\leq 2n+m-2$, contradiction.
This shows that $i'=i''$, and hence, by \cite[Lemma~2.4]{DHI16} $ L(\Theta_{\widehat i})\prec^s_M L(\tilde G\times \tilde G^0_{\widehat {i'}})$.
\hfill$\square$

{\bf Claim 2.} $\Delta(L(\tilde G^0))\prec_{M\bar\otimes M}M\bar\otimes L(\tilde G^0)$.





{\it Proof of Claim 2.} 
Since $G_{j_0,1}^1$ is icc non-amenable, Lemma \ref{lemma.classification} implies that one of the following holds:
\begin{enumerate}
    \item [(i1)] $\Delta(L(G_{j_0,1}^2\times \tilde G_{\widehat {j_0}}\times \tilde G^0))\prec_{M\bar\otimes M} M\bar\otimes L(G_{j',\alpha'}\times \tilde G_{\widehat {j'}} \times \tilde G^0)$ for some $j'\in\{1,\dots,n\}$ and $\alpha'\in\{1,2\}$, or 

    \item [(i2)] $\Delta(L(G_{j_0,1}^2\times\tilde G_{\widehat {j_0}}\times \tilde G^0))\prec_{M\bar\otimes M} M\bar\otimes L(\tilde G \times \tilde G^0_{\widehat i'})$ for some $i\in\{1,\dots,m\}$.
\end{enumerate}

 If (i2) holds, we have $\Delta(L(\tilde G^0 ))\prec_{M\bar\otimes M} M\bar\otimes L(\tilde G\times \tilde G^0_{\widehat i})$. 
 Note that  $\mathcal N_{M\bar\otimes M}(\Delta(L(\tilde G^0)))'\cap M\bar\otimes M = \mathbb C 1$. 
 From Claim 1 it follows that $\Delta(L(\tilde G))\prec M\bar\otimes L(\tilde G\times \tilde G^0_{\widehat i})$
and so we can use \cite[Lemma 2.6]{Is19} to derive that $\Delta(M)\prec_{M\bar\otimes M} M\bar\otimes L(\tilde G\times \tilde G^0_{\widehat i})$. By \cite[Lemma 7.2]{IPV10}, we get that $G_i^0$ is a finite group, contradiction.
 
 Therefore, (i1) holds. By Lemma \ref{lemma.normalizer} it follows that
\begin{enumerate}
    \item [(i1.1)] $\Delta(L(G_{j_0,1}^2\times \tilde G_{\widehat {j_0}}\times \tilde G^0))\prec_{M\bar\otimes M} M\bar\otimes L(\Sigma_{j'}\times \tilde G_{\widehat {j'}} \times \tilde G^0)$, or 

    \item [(i1.2)] $\Delta(L(G_{j_0,1}\times\tilde G_{\widehat {j_0}}\times \tilde G^0))\prec_{M\bar\otimes M} M\bar\otimes L(G_{j',\alpha'}\times \tilde G_{\widehat {j'}} \times \tilde G^0)$.

\end{enumerate}

Lemma \ref{lemma.count} implies that (i1.1) cannot hold, thus (i1.2) must hold. Since $G_{j_0,1}^2$ is icc non-amenable, Lemma \ref{lemma.classification} implies that one of the following holds:
\begin{enumerate}
    \item [(i1.2.1)] $\Delta(L(G_{j_0,1}^1\times \tilde G_{\widehat {j_0}}\times \tilde G^0))\prec_{M\bar\otimes M} M\bar\otimes L(G_{j',\alpha'}^{\beta'}\times \tilde G_{\widehat {j'}} \times \tilde G^0)$ for some $\beta'\in\{1,2\}$, or 

    \item [(i1.2.2)] $\Delta(L(G_{j_0,1}^1\times\tilde G_{\widehat {j_0}}\times \tilde G^0))\prec_{M\bar\otimes M} M\bar\otimes L(G_{j',\alpha'}\times G_{j'',\alpha ''}\times \tilde G_{\widehat {\{j',j''\}}} \times \tilde G^0)$ for some $j''\in \{1,\dots,n\}$ with $j''\neq j'$ and $\alpha''\in\{1,2\}$, or

    \item [(i1.2.3)] $\Delta(L(G_{j_0,1}^1\times\tilde G_{\widehat {j_0}}\times \tilde G^0))\prec_{M\bar\otimes M} M\bar\otimes L(G_{j',\alpha'}\times \tilde G_{\widehat {j'}} \times \tilde G^0_{\widehat{i'}})$ for some $i'\in\{1,\dots,m\}$.

\end{enumerate}

If (i1.2.3) holds, we obtain a contradiction as in (i2) above. If (i1.2.2) holds, then $\Delta(L(\tilde G_{\widehat {j_0}}\times \tilde G^0))\prec_{M\bar\otimes M} M\bar\otimes L(G_{j',\alpha'}\times G_{j'',\alpha ''}\times \tilde G_{\widehat {\{j',j''\}}} \times \tilde G^0)$, and by Lemma \ref{lemma.normalizer} we further derive that $\Delta(L(\tilde G_{\widehat {j_0}}\times \tilde G^0))\prec_{M\bar\otimes M} M\bar\otimes L(\Sigma_{j'}\times \Sigma_{j''}\times \tilde G_{\widehat {\{j',j''\}}} \times \tilde G^0)$. Lemma \ref{lemma.count} implies that $2(n-1)+m \leq 2(n-2)+m$, contradiction. 

Therefore, (i1.2.1) holds. Since $G_{j_0,1}^1$ is icc non-amenable, Lemma \ref{lemma.classification} implies that one of the following holds:

\begin{itemize}
    \item [(a)] $\Delta( \tilde G_{\widehat {j_0}}\times \tilde G^0))\prec_{M\bar\otimes M} M\bar\otimes L( \tilde G_{\widehat {j'}} \times \tilde G^0)$, or

    \item [(b)] $\Delta( \tilde G_{\widehat {j_0}}\times \tilde G^0))\prec_{M\bar\otimes M} L(G_{j',\alpha'}^{\beta'}\times G_{j'',\alpha''}\times\tilde G_{\widehat {\{j',j''\}}} \times \tilde G^0)$, for some $j''\in \{1,\dots,n\}$ with $j''\neq j'$ and $\alpha''\in\{1,2\}$, or

    \item [(c)] $\Delta( \tilde G_{\widehat {j_0}}\times \tilde G^0))\prec_{M\bar\otimes M} M\bar\otimes L(G_{j',\alpha'}^{\beta'}\times \tilde G_{\widehat {j'}} \times \tilde G^0_{\widehat i})$ for some $i\in\{1,\dots,m\}$.
\end{itemize}

If (c) holds, then we get a contradiction as in the case (i2) above. If (b) holds, then Lemma \ref{lemma.normalizer} implies that $\Delta( \tilde G_{\widehat {j_0}}\times \tilde G^0))\prec_{M\bar\otimes M} L(G_{j',\alpha'}^{\beta'}\times \Sigma_{j''}\times\tilde G_{\widehat {\{j',j''\}}} \times \tilde G^0)$ and by Lemma \ref{lemma.count} it follows that $2(n-1)+m \leq 1 + 2(n-2)+m$, contradiction. Therefore, (a) holds. 

It follows that there exists a function $f:\{1,\dots,n\}\to \{1,\dots,n\}$ such that
$$
\Delta(L(\tilde G_{\widehat {k}}\times \tilde G^0))\prec_{M\bar\otimes M} M\bar\otimes L( \tilde G_{\widehat {f(k)}} \times \tilde G^0),
$$
for any $k\in \{1,\dots,n\}$. By \cite[Lemma 2.6]{Is19} and \cite[Lemma 7.2]{IPV10}, we get that $f$ is injective, hence it is bijective. By \cite[Lemma 2.8]{DHI16}, it follows that $\Delta(L(\tilde G^0))\prec_{M\bar\otimes M} M\bar\otimes L( \tilde G^0).$
\hfill$\square$

{\bf Claim 3.} There exist subgroups $\tilde\Theta,\tilde\Theta^0<H$ satisfying
\begin{enumerate}
    \item [(a)] $L(\tilde G)\prec_{M} L(\tilde \Theta)$ and $L(\tilde \Theta)\prec^s_M L(\tilde G)$,

    \item [(b)] $L(\tilde G^0)\prec_{M} L(\tilde \Theta^0)$ and $L(\tilde \Theta^0)\prec_M L(\tilde G^0)$.
\end{enumerate}

{\it Proof of Claim 3.} By Claim 1.(b) we get that $L(\tilde G)\prec^s_M L(\Theta_{\widehat i})$ for any $i\in\{1,\dots,n\}$. By applying repeatedly \cite[Lemma 2.7]{Va10b}, there exist $h_1,\dots,h_{n}\in H$ such that $L(\tilde G)\prec^s_M L(\tilde\Theta)$, where $\tilde\Theta=\cap_{i=1}^n h_i \Theta_{\widehat i} h^{-1}_i$. Note that Claim 1.(b) implies that $L(\tilde\Theta)\prec_{M}^s L(\tilde G\times \tilde G^0_{\widehat i})$, for any $i\in\{1,\dots,n\}$. By applying \cite[Lemma 2.8]{DHI16} we deduce that $L(\tilde\Theta)\prec_{M}^s L(\tilde G)$, proving (a).

The proof of (b) follows directly from Claim 2 and \cite[Theorem 4.2]{Dr20}.  
\hfill$\square$

Finally, Claim 3 allows to apply \cite[Theorem 2.3]{Dr20} (see also \cite[Theorem 6.1]{DHI16}) to derive the conclusion.
\hfill$\blacksquare$

\section{Identification of peripheral structure}\label{section.peripheral}
Let $G = G_1 \times \cdots \times G_n$  be a product of groups in the class $\mathcal{C}_{AFP}$. Fix an index $i \in \{1,\ldots,n\}$ and a subgroup $G_0 < G_i$. The main goal of this section is to build upon the methods developed in \cite{CI17, CD-AD20} to recover the subgroup $G_0 \times G_{\widehat{i}}$ at the von Neumann algebra level; see Theorem~\ref{theorem.peripheral} below.

Throughout this subsection we use the assumption from Notation \ref{notation} with $m=0$ and $M^t=N$.

\begin{theorem}\label{theorem.first.step}
Let $i\in \{1,\dots,n\}$. Then there exist a subgroup $\Theta<H$ with ${\rm QN}_{H}^{(1)}(\Theta)=\Theta$, a unitary $u\in \mathcal U(M)$, projections $r_1,r_2\in \mathcal Z(L(\Theta))$ with $r_1+r_2=1$ and $r_1\neq 0$ such that 
\begin{itemize}
    \item $u L(\Theta)r_1 u^*= ur_1u^* L(G_{i,1}\times G_{\widehat i})^t ur_1u^*$,
    \item $u L(\Theta)r_2 u^* \subset ur_2u^* L(G_{i,2}\times G_{\widehat i})^t ur_2u^*$.
\end{itemize}
\end{theorem}

{\it Proof.} Again, for ease of notation we assume $t=1$. By Lemma \ref{lemma.ultrapower} we have $\Delta(L(G_{i,1}^1\times G_{\widehat i}))\prec_{M\bar\otimes M} M\bar\otimes L(G_{a,b}^c\times G_{\widehat a})$ for some $a\in\{1,\dots,n\}$ and $b,c\in\{1,2\}$. By \cite[Theorem 3.1]{Io11} (see also \cite[Theorem 4.1]{DHI16} and \cite[Theorem 3.3]{CdSS15}), there exists a subgroup $\Delta<H$ with non-amenable centralizer $C_{H}(\Delta)$ such that $L(G_{i,1}^1\times G_{\widehat i})\prec_{M}L(\Delta)$. 

Define $\Theta$ as the subgroup of $H$ generated by ${\rm QN}^{(1)}_{H}({\rm QN}_{H}(\Delta))$. Let $r_2\in L(\Theta)'\cap L(H)$ be maximal projection with the property that $L(\Theta)r_2\prec_{M}^s L(G_{i,2}\times G_{\widehat i})$. Let $r_1=1-r_2$.

Next, from $L(G_{i,1}^1\times G_{\widehat i})\prec_{M}L(\Delta)$,
 \cite[Lemma 2.4]{DHI16} implies that there exists a non-zero projection $r_0\in L(\Delta)'\cap M$ such that $L(G_{i,1}^1\times G_{\widehat i})\prec_{M}L(\Delta)z$ for any non-zero projection $z\in L(\Delta)'\cap M$ with $z\leq r_0$. It follows that $r_0\leq r_1$. Indeed, if this is not the case by \cite[Lemma 3.7]{Va08} we would get $L(G_{i,1}^1)\prec_{L(G_i)}L(G_{i,2})$ which implies that $L(G_{i,1}^1)\prec_{L(G_i)} L(\Sigma_i)$. Since $L(G_{i,1}^1)$ is non-amenable II$_1$ factor and $\Sigma_i$ amenable, we get a contradiction. Thus, 
\begin{equation}\label{z0}
L(G_{i,1}^1\times G_{\widehat i})\prec_{M}L(\Delta)r_1.    
\end{equation}

{\bf Claim 1.} $L(\Theta)r_1\prec_{M}^s L(G_{i,1}\times G_{\widehat i})$.

{\it Proof of Claim 1.} Let $z\in L(\Theta C_H(\Theta))'\cap M$ be a non-zero projection with $z\leq r_1$. By \cite[Lemma 2.4]{DHI16} it is enough to show that $L(\Theta)z\prec_{M} L(G_{i,1}\times G_{\widehat i})$.
Since $C_H(\Theta)$ is non-amenable (see also \cite[Lemma 2.13]{CD-AD20}),  Lemma \ref{lemma.classification} implies that 
     $L(\Theta)z\prec_{M} L(G_{i',\alpha} \times G_{\widehat{i'}})$ for some $i'\in \{1,\dots,n\}$ and $\alpha\in \{1,2\}$.


We need to show that $i=i'$ and $\alpha=1$. Assume $i\neq i'$.
In combination with $L(G_{i,1}^1\times G_{\widehat i})\prec_{M}L(\Delta)$, Lemma \ref{lemma.augm} implies that $\Delta(L(G_{i,1}^1\times G_{\widehat i}))\prec_{M\bar\otimes M} M\bar\otimes L(G_{i',\alpha} \times G_{\widehat{i'}})$. Since $\mathcal N_{M\bar\otimes M}(\Delta(G_{\widehat i})))''\supset \Delta(M)$ and $\Delta(M)'\cap M\bar\otimes M=\mathbb C 1$, it follows from Lemma \ref{lemma.normalizer}, \cite[Lemma 7.2]{IPV10}, \cite[Lemma 2.4]{DHI16} and \cite[Proposition 2.5]{Dr19a}  that $\Delta(L(G_{\widehat i}))\prec^s_{M\bar\otimes M} M\bar\otimes L(\Sigma_i \times \Sigma_{i'} \times G_{\widehat {\{i,i'\}}})$. This contradicts Lemma \ref{lemma.count}. Hence $i=i'$. If we assume $\alpha=2$, it follows that $L(\Theta)z\prec_{M} L(G_{i,2}\times G_{\widehat i})$. This contradicts the maximality of $r_2$, hence $\alpha=1$.
\hfill$\square$

{\bf Claim 2.} There is a unitary $u\in\mathcal U(M)$ satisfying $u L(\Theta)r_k u^* \subset L(G_{i,k}\times G_{\widehat i})$, for any $k\in\{1,2\}$.

{\it Proof of Claim 2.} Let $k\in\{1,2\}$. Note that $L(\Theta)r_k\nprec_{M}L(\Sigma_i\times G_{\widehat i})$. Indeed, otherwise we get a contradiction by applying Lemma \ref{lemma.augm} and Lemma \ref{lemma.count}.
By applying Lemma \ref{lemma.normalizer} to Claim 1, there are unitaries $u_1,u_2\in \mathcal U(M)$ such that $u_k L(\Theta)r_k u_k^* \subset L(G_{i,k}\times G_{\widehat i})$, for any $k\in\{1,2\}$. Since $L(\Sigma_i)$ is a II$_1$ factor, we may assume by \cite[Lemma 4.5]{CdSS17} that $u_2r_2u_2^*\in L(\Sigma_i)$.
Since $L(G_{i,1}\times G_{\widehat i})$ is a II$_1$ factor and $\tau(u_1r_1u_1^*)=\tau(1-u_2r_2u_2^*)$, we may assume (up to replacing $u_1$ with a different unitary) that $u_1r_1u_1^*=1-u_2r_2u_2^*$.
Define $u=u_1r_1+u_2r_2$ and note that $u\in \mathcal U(M)$ satisfies the claim.
\hfill$\square$

From \eqref{z0}, we have $L(G_{i,1}^1\times G_{\widehat i})\prec_{M} uL(\Delta)r_1u^*$ and note that $L(G_{i,1}^1\times G_{\widehat i})\nprec_{M} L(\Sigma_i\times G_{\widehat i})$.
By \cite[Corollary 2.10]{CI17} (see also \cite[Theorem 1.1]{IPP05}) and 
 Claim 2, we derive that $L(G_{i,1}^1\times G_{\widehat i})\prec_{L(G_{i,1}\times G_{\widehat i})} u L(\Delta) r_1 u^*$.  By using \cite[Proposition 2.4]{CKP14}, and its proof we can find projections $p\in L(G_{i,1}^1\times G_{\widehat i})$ and $e\in L(\Delta)$, a non-zero partial isometry $v\in uer_1u^*L(G_{i,1}\times G_{\widehat i})p$, a $*$-isomorphism $\theta: pL(G_{i,1}^1\times G_{\widehat i})p\to C \subset  ue L(\Delta) r_1e u^*$ such that the following properties are satisfied:
\begin{enumerate}
\item [(a)] the inclusion $C \vee (C' \cap  ueL(\Delta) r_1e u^* ) \subset  ueL(\Delta) r_1e u^* $  has finite index in the sense of Pimsner-Popa \cite{PP86},
\item [(b)] $\theta(x) v =v x$ for all $x\in pL(G_{i,1}^1\times G_{\widehat i})p$. 
\end{enumerate} 

Note that $C$, $C' \cap  ueL(\Delta) r_1e u^* $ and $uL(C_{H}(\Delta)) r_1e u^* $ are commuting subalgebras of $L(G_{i,1}\times G_{\widehat i})$. Since $L(C_{H}(\Delta))$ is non-amenable, Lemma \ref{lemma.classification} implies that one of the following holds:
\begin{enumerate}
    \item [(1)] $C \vee (C' \cap  ueL(\Delta) r_1e u^* ) \prec^s_{L(G_{i,1}\times G_{\widehat i})}L(G_{i,1}^{1}\times G_{\widehat i})$, or

    \item [(2)]  $C \vee (C' \cap  ueL(\Delta) r_1e u^* ) \prec_{L(G_{i,1}\times G_{\widehat i})}L(G_{i,1}^{2}\times G_{\widehat i})$ , or

    \item [(3)] $C \vee (C' \cap  ueL(\Delta) r_1e u^* ) \prec_{L(G_{i,1}\times G_{\widehat i})}L(G_{i,1}\times G_{j,\alpha} \times G_{\widehat {i,j}})$ for some $j\neq i$ and $\alpha\in\{1,2\}$.
\end{enumerate}

We show that (1) holds.
On one hand if (2) is true, we get by Lemma \ref{lemma.intertwining} that $L(G_{i,1}^1\times G_{\widehat i})\prec_{M} L(G_{i,1}^{2}\times G_{\widehat i}),$ contradiction. On the other hand, by assuming that (3) is true, Lemma \ref{lemma.intertwining} would imply that $L(G_{i,1}^1\times G_{\widehat i})\prec_{M} L(G_{i,1}\times G_{j,\alpha} \times G_{\widehat {\{i,j\}}})$, which implies that $L(G_j)\prec_{L(G_j)} L(G_{j,\alpha})$, contradiction. Thus, (1) holds.

By Lemma \ref{lemma.count}, we get $C' \cap  ueL(\Delta) r_1e u^* \prec_{M} \mathbb C1$. Thus, one can find a non-zero projection $e_0\in C' \cap  ueL(\Delta) r_1e u^* $ such that $e_0(C' \cap  ueL(\Delta) r_1e u^*) e_0=\mathbb Ce_0$. Therefore, by \cite[Lemma 2.3]{PP86} after compressing the inclusion of (a) by $e_0$, we may assume that  the inclusion $C  \subset  ueL(\Delta) r_1e u^* $  is a finite index inclusion of II$_1$ factors. By \cite[Proposition 1.3]{PP86}, there exist $x_1,\dots,x_m\in ueL(\Delta) r_1e u^*$ such that $ueL(\Delta) r_1e u^* = \sum_{i=1}^m x_i C$. By letting $r= ur_1eu^*$, note that $ueL(\Delta) r_1e u^* \subset rMr$, and hence, 
\begin{equation}\label{eq1}
    \mathcal{QN}_{rMr}(C)''=\mathcal{QN}_{rMr}(ueL(\Delta) r_1e u^* )''.
\end{equation}

{\bf Claim 3.} ${\rm QN}_{H}^{(1)}(\Theta)=\Theta$ and $vv^*L(G_{i,1}\times G_{\widehat i}) vv^*
=vv^* ur_1L(\Theta) r_1 u^*vv^*.
$

{\it Proof of Claim 3.}
Let $l\in \mathcal U(L(G_{i,1}\times G_{\widehat i}))$ such that $lv^*v=v$. Since $v\in uer_1u^*L(G_{i,1}\times G_{\widehat i})p$, note that $vv^*\leq uer_1u^*$, which implies $vv^*ur_1=vv^*uer_1$. Since $v\in L(G_{i,1}\times G_{\widehat i})$, it follows that $vv^* L(G_{i,1}\times G_{\widehat i}) vv^*= vL(G_{i,1}\times G_{\widehat i})v^*$, and hence by Lemmas  \ref{lemma.qn.algebras} and \ref{lemma.qn.groups} we get

\begin{equation}\label{eq1}
\begin{split}
  vv^*L(G_{i,1}\times G_{\widehat i}) vv^* &= lv^*v{\mathcal{QN}_{M}(L(G_{i,1}^{1}\times G_{\widehat i})}''v^*v l^*\\
  & =  {\mathcal{QN}_{lv^*vMv^*vl^*}(lv^*vL(G_{i,1}^{1}\times G_{\widehat i})}v^*v l^*)''\\
  & ={\mathcal{QN}_{vv^*M vv^*}(Cvv^*)}''= vv^*{\mathcal{QN}_{rMr}(C)}'' vv^* \\
  & = vv^*\mathcal{QN}_{rMr}(ueL(\Delta) er_1 u^* )'' vv^*\\
  & = vv^* ur_1e\mathcal{QN}_{M}(L(\Delta) )'' er_1 u^* vv^*   \\
  & = vv^*ur_1e L({\rm QN}_{H}(\Delta)) er_1 u^* vv^*\\
  & = vv^*u L({\rm QN}_{H}(\Delta)) r_1 u^* vv^*.  
\end{split}
\end{equation}

Lemmas \ref{lemma.qn.algebras} and \ref{lemma.qn.groups} applied to \eqref{eq1} imply that 
$$
vv^*L(G_{i,1}\times G_{\widehat i}) vv^*= vv^*u L(\Theta) r_1 u^* vv^*= vv^*u L({\rm QN}_{H}(\Delta)) r_1 u^* vv^*.
$$
From \cite[Lemma 2.2]{CI17} we deduce that $[\Theta: {\rm QN}_{H}(\Delta)]<\infty$. It therefore follows that, ${\rm QN}_{H}^{(1)}(\Theta)$$ = {\rm QN}_{H}^{(1)}({\rm QN}_{H}(\Delta))=\Theta.$
\hfill$\square$

To conclude the proof, note that Claim 3, Lemma \ref{lemma.qn.algebras} and Lemma \ref{lemma.qn.groups}  show that $\mathcal{QN}_{M}^{(1)}(L(\Theta)r_1)=L(\Theta)r_1$.
Note  that
$u L(\Theta)r_1 u^* \subset L(G_{i,1}\times G_{\widehat i})$, $
vv^* uL(\Theta) r_1 u^*vv^*=vv^*L(G_{i,1}\times G_{\widehat i}) vv^*$ and $L(G_{i,1}\times G_{\widehat i})$ is a II$_1$ factor. By \cite[Lemma 2.6]{CI17},  we get that $
 uL(\Theta) r_1 u^*= ur_1u^* L(G_{i,1}\times G_{\widehat i}) ur_1u^*$.
\hfill$\blacksquare$

The following theorem follows closely the arguments used in \cite[Theorem 3.2]{CI17}) (see also \cite[Theorem 9.2]{CD-AD20}) and \cite[Proposition 4.1]{CI17}. We provide some details for the convenience of the reader.

\begin{theorem}\label{theorem.peripheral}
Let $i\in\{1,\dots,n\}$. Then there exist subgroups $\Theta_1,\Theta_2<H$, unitaries $u_1,u_2\in \mathcal U(M^t)$ such that $u_kL(\Theta_k)u_k^*=L(G_{i,k}\times G_{\widehat i})^t$, for any $k\in\{1,2\}$.
\end{theorem}

{\it Proof.} By Theorem \ref{theorem.first.step}, for any $k\in\{1,2\}$ there exists a subgroup $\Theta_k<H$ with ${\rm QN}_{H}^{(1)}(\Theta_k)=\Theta_k$, a unitary $u_k\in \mathcal U(M)$, projections $r_k^1,r_k^2\in \mathcal Z(L(\Theta_k))$ with $r_k^1+r_k^2=1$ and $r_k^1\neq 0$ such that
\begin{enumerate}
    \item $u_k L(\Theta_k)r_k^1 u_k^*= u_kr_k^1u_k^* L(G_{i,k}\times G_{\widehat i})^t u_kr_k^1u_k^*$,
    \item $u_k L(\Theta_k)r_k^2 u_k^* \subset u_kr_k^2u_k^* L(G_{i,3-k}\times G_{\widehat i})^t u_kr_k^2u_k^*$.
\end{enumerate}

If $r_k^1=1$ for any $k\in\{1,2\}$, then we are done. Without loss of generality, let us assume that $r_1^1\neq 1$ and derive a contradiction. Since $r_1^2\neq 0$, condition (2) implies that $ L(\Theta_1) \prec_{M^t} L(G_{i,2}\times G_{\widehat i})^t$. Condition (1) together with \cite[Lemma 2.4]{DHI16} implies that $L(G_{i,2}\times G_{\widehat i})^t\prec_{M^t}^s L(\Theta_2)$. By \cite[Lemma 3.7]{Va08} it follows that  $L(\Theta_1)\prec_M L(\Theta_2)$, and by \cite[Lemma 2.2]{CI17} there exists $h\in H$ such that $[\Theta_1: \Theta_1 \cap h \Theta_2 h^{-1}]<\infty$. Up to replacing $\Theta_2$ by $h\Theta_2 h^{-1}$ we may assume that $[\Theta_1:\Theta]<\infty,$ where $\Theta=\Theta_1\cap\Theta_2$.

We continue by proving that $r_1^1 r_2^1=r_1^2 r_2^2=0$. Indeed, if this does not hold, then by \cite[Lemma 2.7]{Va10b} there exists $g\in G_i$ such that $L(\Theta)\prec_{M^t} L((G_{1,1}\cap g G_{1,2} g^{-1})\times G_{\widehat{i}})^t$, thus, $L(\Theta)\prec_{M^t} L( \Sigma_i \times G_{\widehat{i}})^t$.
By combining with condition (1), the fact that $[\Theta_1:\Theta]<\infty$ and Lemma \ref{lemma.augm}, we get that $\Delta(L(G_{i,1}\times G_{\widehat i})^t)\prec_{M^t\bar\otimes M^t}M^t\bar\otimes L( \Sigma_i \times G_{\widehat{i}})^t$. 
This contradicts Lemma \ref{lemma.count}.

Since $r_1^1 r_2^1 = (1-r_1^1)(1-r_2^1)=0$,
it follows that $r_1^1+r_2^1=1$, and hence, $r_2^2=r_1^1$. By using (1) for $k=1$ and (2) for $k=2$, it follows that 
\begin{equation}\label{eq.11}
u_1 L(\Theta)r_1^1 u_1^*\subset
u_1 L(\Theta_1)r_1^1 u_1^*= u_1r_1^1u_1^* L(G_{i,1}\times G_{\widehat i})^t u_1r_1^1u_1^* 
\end{equation}
and
\begin{equation}\label{eq.12}
u_2 L(\Theta)r_1^1 u_2^*\subset
 u_2 L(\Theta_2)r_1^1 u_2^* \subset u_2r_1^1u_2^* L(G_{i,1}\times G_{\widehat i})^t u_2r_1^1u_2^*.    
\end{equation}

Since $L(\Theta)\nprec_{M^t}L(\Sigma_i\times G_{\widehat i})^t$ and 
$
u_1 r_1^1 u_2^*(u_2 L(\Theta) r_1^1 u_2^*) = (u_1 L(\Theta) r_1^1 u_1^*)u_1r_1^1u_2^*,
$
 by \eqref{eq.11}, \eqref{eq.12} and \cite[Corollary 2.10]{CI17}, we get that $u_1r_1^1u_2^*\in L(G_{i,1}\times G_{\widehat i})^t$. By \eqref{eq.11} and \eqref{eq.12} we further get $L(\Theta_2)r_1^1\subset L(\Theta_1)r_1^1$. This implies that $\Theta_2\subset\Theta_1$. Thus, $\Theta=\Theta_2$, and hence, $[\Theta_1:\Theta_2]<\infty$. Since ${\rm QN}_{H}^{(1)}(\Theta_k)=\Theta_k$, for any $k\in\{1,2\}$, it follows that $\Theta_1=\Theta_2=\Theta$.

Finally, by proceeding as in the proof of \cite[Proposition 4.1]{CI17}, condition (1)  contradicts the icc condition of the group $\Sigma_i$. 
\hfill$\blacksquare$


\begin{remark}\label{remark.Wsup.amplification} We note that the above result also shows that if $G_1=G_{1,1}*_{\Sigma_1}G_{1,2}\in \mathcal C_{AFP}^0$ and $L(G_1)^t\cong L(H_1)$ for some $t>0$ and countable group $H_1$, then $t=1$. Indeed, by assuming $M=L(G_1)$ and $M^t=L(H_1)$, Theorem \ref{theorem.peripheral} shows that there exist a unitary $u\in \mathcal U(L(M^t))$ and a subgroup $\Theta_1<H_1$ satisfying $uL(\Theta_1)u^* \cong L(G_{1,1})^t$. By \cite[Corollary B]{CD-AD20}, it follows that $t=1$.
\end{remark}





\section{Proofs of the main results}\label{section.proofs}

\subsection{Proof of Theorem \ref{B}}
By Theorem \ref{theorem.peripheral}, there is $u\in\mathcal U(M)$ and a subgroup $\Theta_n<H$ such that $uL(\Theta_n)u^*=L(G_{\widehat n}\times G_{n,1})^t$. By Theorem \ref{theorem.product.rigidity}, there exist a subgroup $\Lambda_{ n }<H$ and a unitary $w\in \mathcal U(M)$ and $t'>0$ such that $wL(\Lambda_n) w^* = L(G_{\widehat n})^{t'}$. By \cite[Theorem 4.6]{CdSS17} and \cite[Theorem 4.7]{CdSS17}, it follows that there exist a product decomposition $H=H_{\widehat n}\times H_n$, a positive number $t_n>0$ and a unitary $\omega\in \mathcal U(M)$ such that $\omega L(H_{\widehat n})  \omega^* = L(G_{\widehat n})^{t_n}$ and $\omega L(H_{ n})  \omega^* = L(G_{ n})^{1/{t_n}}$. The theorem follows by repeating the same arguments finitely many times.
\hfill$\blacksquare$

\subsection{Proof of Theorem \ref{A}}
The result follows by combining Theorem \ref{B}, Remark \ref{remark.Wsup.amplification} and \cite[Theorem C]{CD-AD20}.
\hfill$\blacksquare$

\subsection{Proof of Corollary \ref{corA}}
By Lemma \ref{lemma.trivial.radical}, the amenable radical of $G$ is trivial and so $C^*_r(G)$ has a unique trace by \cite[Theorem~1.3]{BKKO14}. It follows that the isomorphism $\theta$ extends to an isomorphism of the group von Neumann algebras. The result now follows from Theorem \ref{A}.
\hfill$\blacksquare$

\end{document}